\documentclass[12pt,oneside,reqno]{amsart}
\usepackage{}
\usepackage{amssymb}

\usepackage{bbm}
\usepackage{cases}
\usepackage{amsmath}
\usepackage{graphicx}
\usepackage{mathrsfs}
\usepackage{stmaryrd}
\usepackage{color}
\usepackage{soul}
\usepackage[dvipsnames]{xcolor}
\usepackage{amsfonts}
\usepackage{comment}
\usepackage{enumerate,amsmath,amssymb,amsthm}
\numberwithin{equation}{section}

\newcommand{\be}{\begin{eqnarray}}
\newcommand{\mE}{\end{eqnarray}}
\newcommand{\ce}{\begin{eqnarray*}}
\newcommand{\de}{\end{eqnarray*}}
\newtheorem{theorem}{Theorem}[section]
\newtheorem{lemma}[theorem]{Lemma}
\newtheorem{proposition}[theorem]{Proposition}
\newtheorem{corollary}[theorem]{Corollary}

\theoremstyle{definition}
\newtheorem{definition}[theorem]{Definition}
\newtheorem{remark}[theorem]{Remark}
\newtheorem{example}[theorem]{Example}

\newtheorem{assumption}[theorem]{Assumption}

\def\p{\partial}

\def\R{\mathbb{R}}
\def\[{{\Big[}}
\def\]{{\Big]}}
\def\<{{\langle}}
\def\>{{\rangle}}
\def\({{\Big(}}
\def\){{\Big)}}

\def\bx{{\mathbf{x}}}

\def\dif{{\mathord{{\rm d}}}}

\def\={&\!\!=\!\!&}
\def\bt{\begin{theorem}}
\def\et{\end{theorem}}
\def\bl{\begin{lemma}}
\def\el{\end{lemma}}
\def\br{\begin{remark}}
\def\er{\end{remark}}

\def\bd{\begin{definition}}
\def\ed{\end{definition}}
\def\bp{\begin{proposition}}
\def\ep{\end{proposition}}
\def\bc{\begin{corollary}}
\def\ec{\end{corollary}}
\def\bx{\begin{example}}
\def\ex{\end{example}}

\def\cB{{\mathcal B}}

\def\cF{{\mathcal F}}

\def\cX{{\mathcal X}}

\def\mE{{\mathbb E}}

\def\mH{{\mathbb H}}

\def\mL{{\mathbb L}}

\def\mN{{\mathbb N}}

\def\mP{{\mathbb P}}

\def\mR{{\mathbb R}}

\def\mY{{\mathbb Y}}

\def\sH{{\mathscr H}}

\def\geq{\geqslant}
\def\leq{\leqslant}

\allowdisplaybreaks

\begin{document}
	\title{The perfection of local semi-flows and local random dynamical systems with applications to
	SDEs}
	
	\date{}

\author{Chengcheng Ling,  Michael Scheutzow and Isabell Vorkastner}

\address{Chengcheng Ling:
Technische Universit\"at Berlin,
Fakult\"at II, Institut f\"ur Mathematik,
10623 Berlin, Germany
\\
Email: ling@math.tu-berlin.de
 }

\address{Michael Scheutzow:
Technische Universit\"at Berlin,
Fakult\"at II, Institut f\"ur Mathematik,
10623 Berlin, Germany
\\
Email: ms@math.tu-berlin.de
}

\address{Isabell Vorkastner:
Technische Universit\"at Berlin,
Fakult\"at II, Institut f\"ur Mathematik,
10623 Berlin, Germany
\\
Email: vorkastn@math.tu-berlin.de
}	
\begin{abstract}

We provide a rather general perfection result for crude local  semi-flows taking values in a Polish space showing that a crude semi-flow has a modification which is  a (perfect) local semi-flow which is invariant under a suitable metric dynamical system. Such a (local) semi-flow induces a (local) random dynamical system.    
Then  we show that this result can be applied to several classes of stochastic differential equations
driven by semimartingales with stationary increments such as equations with locally monotone coefficients and equations with singular drift. For these examples it was previously unknown whether they generate a 
(local) random dynamical system or not.
		\bigskip
		

	\end{abstract}

	\maketitle
	
\section{Introduction}	In this paper, we investigate the relation between (local) semi-flows and (local) random dynamical systems. 
Let us start with a very simple set-up (without measurability or topological assumptions) to illuminate this relation. To ease notation, we define
$$
\Delta:=\{(s,t)\in \R^2:\,0 \leq s \leq t\}.
$$
Let $\Omega$ be a non-empty set, $(H,\circ)$ a semi-group with identity element $e_H$ and let $\theta_t$, $ t\in \R$ be a family of 
maps from $\Omega$ to itself such that $\theta_0=\mathrm{id}$ and $\theta_{t+s}=\theta_t\theta_s$ for $s,t \in \R$. Let $\phi:\Delta \times \Omega \to H$ and $\varphi:[0,\infty) \times \Omega \to H$ be maps
and consider the following properties which $\phi$ and $\varphi$ may or may not satisfy.
\begin{itemize}
    \item [(i)] $\phi_{s,u}(\omega)=\phi_{t,u}(\omega)\circ \phi_{s,t}(\omega)$ for all $0\leq s \leq t \leq u$, $\omega \in \Omega$,
    \item [(ii)] $\phi_{s,s}(\omega)=e_H$ for all $s \geq 0$, $\omega \in \Omega$,
    \item [(iii)] $\phi_{s,s+t}(\omega) = \phi_{0,t}(\theta_s\omega)$,  $s, t \geq 0$, $\omega \in \Omega$,
    \item [(iv)] $\varphi_{t+s}(\omega)= \varphi_{t}(\theta_s\omega) \circ \varphi_{s}(\omega)$, $s,t \geq 0$, 
    $\omega \in \Omega$,
    \item [(v)] $\varphi_0(\omega)=e_H$, $\omega \in \Omega$.
\end{itemize}
Then it is straightforward to check that if $\phi$ satisfies (i), (ii) and (iii), then 
$\varphi_t(\omega):=\phi_{0,t}(\omega)$, $t \geq 0$ satisfies (iv) and (v). Conversely, if 
$\varphi$ satisfies (iv) and (v), then $\phi_{s,t}(\omega):=\varphi_{t-s}(\theta_s\omega)$, $(s,t) \in \Delta$ satisfies (i), (ii) and (iii). Properties (iv) and (v) are often referred to as the 
{\em cocycle property} and the relation just described shows that there is a one-to-one correspondence between semi-flows $\phi$ (satisfying (i), (ii) and (iii)) and cocycles $\varphi$. 

During the past decades, random dynamical systems (RDS), introduced by L.~Arnold (\cite{A}), have been studied in great  detail. Many examples of  RDS are generated by stochastic differential equations (SDEs) and it is therefore of interest to show that large classes of SDEs actually generate an RDS. By this we mean that the SDE has a solution map $\phi$ which enjoys properties (i), (ii) and (iii) and hence generates a cocycle $\varphi$ as outlined above (and satisfies further properties). We will provide precise definitions at the beginning of the following section.

If one wants to verify that an SDE (possibly infinite dimensional, e.g.~an SPDE or a stochastic delay equation) generates a semi-flow $\phi$, then the problem comes up that in most cases solutions of an SDE are unique only up to sets of measure 0 and therefore, a solution map $\phi$ 
will typically satisfy properties (i)-(iii) only up to sets of measure 0 and the exceptional sets 
may depend on $s,t,u$ and the initial condition $x$. The first issue is to find a modification of the 
map $\phi$ which satisfies at least properties (i) and (ii) and which is continuous with respect to the initial condition $x$. In many cases such a modification can be found using Kolmogorov's continuity theorem \cite{K, SS}.  If the SDE has time homogeneous coefficients, $\theta$ is a metric dynamical system (see Definition \ref{MDS})  and the SDE is driven by a process with stationary increments which is represented via $\theta$, then it is usually not hard to show that property (iii) holds up to a set of measure 0 which may however depend on $s$ (one can usually get rid of the dependence of the exceptional null sets on $t$ by (right) continuity properties of $t \mapsto \phi_{s,t}$).  An important task is then to show that there exists a modification of $\phi$ such that (iii) holds without exceptional null sets.
We will provide a {\em perfection} statement of this kind in Proposition \ref{Proposition_perfection_theorem}. We will formulate the result in such a way that it also applies to 
local semi-flows, i.e.~semi-flows which are only defined up to a random time, e.g.~the blow-up time of 
the SDE. Instead of performing the perfection on the level of semi-flows, one could also first 
go from a crude semi-flow to a crude cocycle (defined like $\varphi$ above) and then find a modification 
of $\varphi$ which is a (perfect) cocycle. This approach was taken in \cite{AS,MoS,MS}. Here we follow 
the approach in \cite{KS} by perfecting a crude semi-flow. Our result is more general than \cite{KS} 
since we also cover local semi-flows which require a different treatment at some places.

The paper is organized as follows. We first introduce the basic concepts. Then we formulate the mentioned perfection result for crude local semi-flows and show how a crude semi-flow generates a local RDS (Theorem \ref{existence_RDS_flow}). In Section 3 we provide conditions on the semi-flow generated by an SDE on $\R^d$ which guarantee that it generates a local RDS and finally, in Section 4, we provide two classes of examples: SDEs with locally monotone coefficients and singular SDEs on $\R^d$ both driven by
Brownian motion.

\section{Main results}
We start by defining the concepts of (local) semi-flows and (local) RDS. 
We will always assume that the state space is a Polish space $X$, i.e.~a separable completely metrizable topological space and we denote its Borel $\sigma$-field 
by $\cX$. In the following, all random processes will be defined on a given probability space
$(\Omega,\cF,\mP)$.
\begin{definition}
    A local semi-flow $(\phi, \Theta)$
    on a Polish space $X$ consists of two  measurable maps
    \begin{align*}
        \Theta : [0,\infty) \times X \times \Omega \rightarrow (0, \infty]
    \end{align*}
    and
    \begin{align*}
        \phi: \left\{ (s,t,x,\omega) \in [0, \infty)^2\times X \times \Omega :  s \leq t < \Theta (s,x,\omega) \right\} \rightarrow X
    \end{align*}
    such that, for each $\omega \in \Omega$,
    \begin{enumerate}
        \item[(i)] $(s,x) \mapsto \Theta(s,x)$ is lower semi-continuous, 
        \item[(ii)] $s \mapsto \phi_{s,t}(x)$ is right continuous on $[0,t)$ for every $x \in X$ and $0 \leq t < \Theta (s,x) $, 
        \item[(iii)] $t \mapsto \phi_{s,t}(x)$ is right continuous on $[s,\Theta(s,x))$ for every $x \in X$ and $s \geq 0$, 
        \item[(iv)] for all $0\leq s \leq t < u$ and $x \in X$, we have $u < \Theta(s,x)$ iff both $t< \Theta(s,x)$ and $u<\Theta(t,\phi_{s,t}(x))$.
        In this case the following identity holds
        \begin{align*}
            \phi_{s,u}(x) = \phi_{t,u} (\phi_{s,t} (x)).
        \end{align*}
        \item[(v)] $\phi_{s,s} (x) = x$ for all $x \in X$ and $s\geq 0$.
    \end{enumerate}
    We call $(\phi,\Theta)$ a {\em continuous local semi-flow} if, in addition, 
    $$
    x \mapsto \phi_{s,t}(x) \mbox{ is continuous for every } 0 \leq s \leq t < \Theta (s,x). 
    $$
    We call $(\phi,\Theta)$ a (\emph{global}) \emph{semi-flow} if $\Theta(s,x,\omega)= \infty$ for all $s\geq 0, x \in X$ and $\omega \in \Omega$.
\end{definition}

Next, we define the concepts of a metric dynamical system and a (local) random dynamical system.

\begin{definition}\label{MDS}
A  \emph{metric dynamical system} (MDS for short) 
$\theta=(\Omega,\mathcal{F},\mP,\{\theta_t\}_{t\in\mathbb{R}})$ is a
probability space $(\Omega,\mathcal{F},\mP)$ with a family of measure preserving transformations $\{\theta_t:\Omega\rightarrow\Omega,t\in\mathbb{R}\}$ such that
     \begin{itemize}
     \item[(1)] $\theta_0=\mathrm{id}, \theta_t\circ\theta_s=\theta_{t+s}$ for all $t,s\in\mathbb{R}$;
     \item[(2)] the map $(t,\omega)\mapsto \theta_t\omega$ is measurable and $\theta_t\mP=\mP$ for all $t\in\mathbb{R}$.
     \end{itemize}
 \end{definition}

 \begin{definition}[local RDS]
 A \emph{local random dynamical system} (RDS)  $(\theta, \varphi,\tau)$
    on a Polish space $X$ consists of an MDS $\theta$ and measurable maps $\tau: X \times \Omega \to (0,\infty]$ and 
    $$\varphi:\left\{ (s,x,\omega)\in [0,\infty) \times X \times \Omega: s < \tau(x,\omega)  \right\}\rightarrow X$$
    such that, for each $\omega \in \Omega$, 
    \begin{itemize}
     \item[(1)] $x \mapsto \tau(x)$ is lower semi-continuous,
     \item[(2)] $t \mapsto \varphi_t(x)$ is right continuous on $[0,\tau(x))$ for every $x \in X$,
     \item[(3)] $\tau(\varphi_s(x,\omega),\theta_s\omega)+s=\tau(x,\omega)$ whenever $s < \tau(x,\omega)$,
     \item[(4)] $\varphi$ satisfies the following (perfect) {\em cocycle property}:
     \begin{align}\label{pcocloc}\varphi_0(.,\omega)=\mathrm{id},\quad \varphi_{t+s}(x,\omega)=\varphi_t(\varphi_s(x,\omega),\theta_s\omega)
     \end{align}
     for all $t,s\geq0$, $x \in X$, and $\omega\in\Omega$ for which $t+s < \tau(x,\omega)$. \\
     
     \end{itemize}
         We call $(\theta,\varphi,\tau)$ a {\em continuous local RDS} if, in addition, for each $t \ge 0$,
    $$
    x \mapsto \varphi_{t}(x) \mbox{ is continuous on } \left\{\tau (x)>t\right\}. 
    $$
    We call $(\theta,\varphi,\tau)$ a (\emph{global}) \emph{RDS} if $\tau(x,\omega)= \infty$ for all $x \in X$ and $\omega \in \Omega$.
 \end{definition}
 
 \begin{remark}
 If $(\theta, \varphi,\tau)$ is a local RDS, then we define
 $$D(\omega):=\left\{(t,x)\in [0,\infty)\times X: t < \tau(x,\omega) \right\}$$
and, for $x \in X$,
$$
D(x,\omega):=\left\{t \ge 0:(t,x) \in D(\omega)\right\}. 
$$
It is easy to check that $D(\omega)$ is an open subset of $[0,\infty) \times X$, 
$D(x,\omega)$ is an open interval of $[0,\infty)$ which contains 0 and that 
$D:=\left\{(t,x,\omega)\in [0,\infty)\times X \times \Omega: (t,x)\in D(\omega) \right\}$ is measurable. 
Further, for each $x,\omega,s$ we have
$$
D(x,\omega)= D(\varphi_s(x,\omega),\theta_s\omega)+s.
$$

Conversely, one can start from a set $D$ with these properties and define $\tau(x,\omega):=\inf\{t \ge 0: t \notin D(x,\omega)\}$. We omit further details and refer the reader to \cite[p.11]{A} for this approach to define local RDS, see also \cite[p.513]{FGS}.
\end{remark}
 
For a Polish space $X$, we define $\bar{X}:=X \cup \{\partial\}$ as the disjoint union of $X$ and an additional point $\partial$ (sometimes called {\em coffin state}). Then $\bar{X}$ is also a Polish space. For $f:(0,\infty) \to \bar{X}$ and $x \in \bar{X}$ we write $x=\text{ess}\lim_{h \downarrow 0}f(h)$ if 
there exist $\varepsilon>0$ and a Lebesgue null set $N\subset (0,\varepsilon]$ such that 
$x=\lim_{h \downarrow 0,h \notin N}f(h)$ (see \cite[p.105]{DM} for a similar definition).
For a probability measure $\mP$ on $(\Omega,\cF)$ we denote the corresponding inner measure on 
$\Omega$ by $\mP_*$.

The following proposition and its proof are similar to \cite[Theorem 4]{KS} which is, however,  formulated for $X$-valued processes $\phi$. Here, the state space is $\bar{X}$ and in  general the 
map $s \mapsto \phi_{s,t}(x,\omega)$ will not be right continuous at $s \in [0,t)$ when  $\phi_{s,t}(x,\omega)=\partial$, so assumption (iv') in  \cite[Theorem 4]{KS} will not hold in that case.
Note that this problem cannot be overcome by equipping $\bar{X}$ with the Alexandrov's one-point compactification even if $X$ is locally compact.
\begin{proposition}
    \label{Proposition_perfection_theorem}
    Let $X$ be a Polish space, let $(\Omega,\mathcal{F},\mathbb{P},\left\{\theta_t\right\}_{t \in \R})$ be a metric dynamical system and define $\bar{X}$ as above.
    Assume that 
    $\phi : \Delta  \times \bar{X} \times \Omega \rightarrow \bar{X}$ is a measurable map satisfying
    \begin{enumerate}
        \item[(i)] $\partial$ is an absorbing state, i.e. $\phi_{s,t}(\partial,\omega) = \partial$ for all $0\leq s \leq t$, $\omega \in \Omega$,
        \item[(ii)] $\phi_{s,u} (x,\omega) = \phi_{t,u} ( \phi_{s,t}(x,\omega), \omega)$ for all $x \in \bar{X}, \omega \in \Omega, 0 \leq s \leq t \leq u $,
        \item[(iii)] for every $s \geq 0 $ there exists a null set $N_s$ such that
        \begin{align*}
            \phi_{s,t} (x,\omega) = \phi_{0,t-s} (x,\theta_s \omega)
        \end{align*}
        for all $x \in \bar{X}, \omega \notin N_s$ and $t \geq s$,
        \item[(iv)] for every $t > 0, \omega \in \Omega, x \in X$ and $s_0\in [0,t)$, the map $s \mapsto \phi_{s,t}(x,\omega)$ is right continuous at $s_0$ if $\phi_{s_0,t}(x,\omega) \in X$,
        \item[(v)] $t \mapsto \phi_{s,t}(x,\omega)$ is right continuous on $[s, \infty)$ for every $s \geq 0, \omega \in \Omega, x \in \bar{X}$.
    \end{enumerate}
    Then, there exists a measurable map  $\tilde{\phi} : \Delta \times \Omega \times \bar{X} \rightarrow \bar{X}$ which satisfies (i), (ii), (iv), (v) and
    \begin{enumerate}
        \item [(iii')] 
        \begin{align*}
            \tilde{\phi}_{s,t} (x,\omega) = \tilde{\phi}_{0,t-s} (x,\theta_s \omega)
        \end{align*}
        for all $x \in \bar{X}, \omega \in \Omega$ and $t \geq s \geq 0$.
        \item [(vi)] For each $s\geq 0$, the processes $\tilde {\phi}_{s,.}(.)$ and  $\phi_{s,.}(.)$ are indistinguishable, i.e.~$\mP_* ( \tilde{\phi}_{s,t} (x) = \phi_{s,t} (x) \text{ for all } t \ge 0 \text{ and } x \in \bar{X}) =1  $. 
        \item [(vii)] If $\phi_{s,s}(x,\omega)=x$ for all $s\geq 0, x \in \bar{X}, \omega \in \Omega$, then the same holds for  $\tilde{\phi}$.
        \item [(viii)] If $\phi_{s,t}(x,\omega) \in X$   for all $(s,t) \in \Delta, x \in X, \omega \in \Omega$, then the same holds for  $\tilde{\phi}$.
        \item [(ix)] If $\phi_{s,t}(.)$ is continuous in $s$ and/or $t$ for every choice of the other variables, then the same holds for $\tilde{\phi}$.
        \item[(x)] If, for some $x_0 \in X$,   $x \mapsto \phi_{s,t}(x,\omega) $ is continuous at $x_0$ for every $(s,t) \in \Delta, \omega \in \Omega$, then the same holds for $\tilde{\phi}$.
    \end{enumerate}
\end{proposition}

\begin{proof}
    Define 
    \begin{align*}
        M = \left\{ (s,\omega) \in [0,\infty) \times \Omega : \phi_{s,s+t}(x,\omega) = \phi_{0,t}(x,\theta_s \omega) \text{ for all } x \in \bar{X}, t \geq 0 \right\}.
    \end{align*}
    The complement of $M$ in $[0, \infty ) \times \Omega$ is the projection of the set
    \begin{align*}
        A = \left\{(s, \omega, t, x) : \phi_{s,s+t} (x,\omega) \neq \phi_{0,t} (x,\theta_s \omega) \right\} \in \cB ([0, \infty)) \otimes \cF \otimes  \cB ([0, \infty)) \otimes \cB ( \bar{X})
    \end{align*}
    onto the first two components. By the projection theorem \cite[Proposition 8.4.4]{C},
    $M$ is measurable w.r.t. the completion of $\cB ([0, \infty)) \otimes \cF$ w.r.t. $\lambda \otimes \mP$, where $\lambda$ denotes Lebesgue measure on $[0,\infty)$. Hence, there exist sets $M_1,M_2 \in \cB ([0, \infty)) \otimes \cF$ such that $M_1 \subset M \subset M_2$ and $\lambda \otimes \mP (M_2 \setminus M_1) =0$. By assumption (iii) $M_2$ has full measure and therefore $M_1$ as well. \\
    Define further
    \begin{align*}
        \tilde{\Omega}= \left\{ \omega \in \Omega : (s+r, \theta_{-r} \omega) \in M_1 \text{ for } \lambda \otimes \lambda \text{ -a.a } (s,r)\in \mR ^2 \text{ such that } s+r \geq 0\right\}.
    \end{align*}
    Invariance of $\theta$ under $\mP$ and Fubini's theorem imply $\tilde{\Omega} \in \cF$ and $\mP (\tilde{\Omega}) =1$. Note that $\tilde{\Omega}$ is invariant under $\theta_u$ for every $u \in \R$. \\
    Define for $s \geq 0,\, t>0$
\begin{align*}
        \tilde{\phi}_{s,s+t}(x,\omega) = 
        \begin{cases}
            \text{ess}\lim_{h \downarrow 0}&\hspace{-.3cm} \phi_{0,t-h} (x,\theta_{s+h}\omega), 
               \\ & \omega \in \tilde{\Omega} \text{ and } 
                \lambda ( r \in (-s, \infty) :
                \phi_{s+r,s+r+t} (x,\theta_{-r} \omega) \in X  ) > 0 \\
            \partial, & \omega \in \tilde{\Omega} \text{ and }  \lambda ( r \in (-s, \infty) :
                \phi_{s+r,s+r+t} (x,\theta_{-r} \omega) \in X  ) = 0 \\
            x, & \omega \notin \tilde{\Omega}.
        \end{cases}
    \end{align*}    
    To see that the essential limit exists, observe that if $s,h \ge 0$, $r \geq -s$, $(s+h+r, \theta_{-r} \omega) \in M$, then 
    \begin{align}
        \label{proof_proposition_esslim}
        \phi_{0, t-h} (x,\theta_{s+h} \omega) = \phi_{s+h+r,s+t+r} (x,\theta_{-r} \omega)
    \end{align}
    for every $x \in \bar{X}$ and $t\ge h$.
    For fixed $s \geq 0$ and $\omega \in \tilde{\Omega}$, equation \eqref{proof_proposition_esslim} holds for  $\lambda \otimes \lambda $-a.a. $(r,h) \in (-s,\infty) \times [0,\infty)$ and all
    $x \in \bar{X}$ and $t\ge h$. Fix $s\geq 0$, $t>0$, $x \in \bar{X}$ and assume that the first case of the definition of $\tilde{\phi}$ applies. Then there exists an $r \in (-s, \infty)$ such that equation \eqref{proof_proposition_esslim} holds for $\lambda$-a.a. $h \in [0,t]$ and 
    $\phi_{s+r,s+t+r} (x,\theta_{-r} \omega) \in X$. Assumption (iv) then implies the existence of the essential limit and also that 
    \begin{align}
        \label{proof_proposition_nullset}
        \tilde{\phi}_{s,s+t} (x,\omega) = \phi_{s+r,s+r+t} (x,\theta_{-r}\omega)
    \end{align}
    for all $s\geq 0,t>0, \omega \in \tilde{\Omega}$  and $\lambda$-a.a. $r \in (-s,\infty)$ where the exceptional null set may depend on $s,t,x$ and $\omega$. If the second case of the definition of $\tilde{\phi}$ applies, then \eqref{proof_proposition_nullset} holds as well (both sides are $\partial)$. 
    Finally, we define
    \begin{align*}
        \tilde{\phi}_{s,s} (x,\omega) = 
        \begin{cases}
            \lim_{h \downarrow 0} \tilde{\phi}_{s,s+h} (x,\omega), 
                & \omega \in \tilde{\Omega} \\
            x, & \omega \notin \tilde{\Omega}.
        \end{cases}
    \end{align*}
    The limit in the first case exists by \eqref{proof_proposition_nullset} and (v). Therefore, for $\omega \in \tilde{\Omega}$, \eqref{proof_proposition_nullset} holds for all $(s,t)\in \Delta$ and all $x \in \bar{X}$.\\
    We claim that $\tilde{\phi}$ satisfies all required properties.
    
    To see that $\tilde{\phi}$ is measurable we use the fact that $\cB(\bar X)$ is countably generated and separates point and therefore, $\big(\bar X,\cB(\bar X)\big)$ can be embedded in $\big([0,1],\cB\big([0,1]\big)\big)$ as a measurable space, \cite[p.194]{Z}. Then
    $$
    \tilde{\phi}_{s,s+t}(x,\omega)=
        \begin{cases}
            \int_0^1\phi_{s+r,s+r+t} (x,\theta_{-r}\omega)\,\mathrm{d} r, &\omega \in \tilde{\Omega},\\ 
             x, & \omega \notin \tilde{\Omega},
        \end{cases}    
    $$
    and measurability of $\tilde{\phi}$ follows from the fact that $\tilde{\Omega} \in \cF$, measurability of $\phi$ and Fubini's theorem. 
    \begin{enumerate}
        \item[(i)] follows directly by assumption (i) and the definition of $\tilde{\phi}$,
        \item[(ii)]  follows by (i) and \eqref{proof_proposition_nullset},
        \item[(iii')] follows from \eqref{proof_proposition_nullset} and the fact that $\tilde{\Omega}$ is invariant under $\theta$ ,
        \item[(iv)] is clear when $\omega \notin \tilde{\Omega}$. If $\omega \in \tilde{\Omega}$, $0\leq s<t$, $x \in X$, $\tilde{\phi}_{s,t}(x,\omega)\in X$ and $s_n \downarrow s$ such that $s_1<t$, then, due to assumption (iv), there 
        exists a Lebesgue null set $N$ (possibly depending on $s$, the sequence $(s_n)$, $\omega$, $x$ and $t$) such that for 
        $r \notin N$ we have
        $$
        \tilde{\phi}_{s,t}(x,\omega)=\phi_{s+r,t+r}(x,\theta_{-r}\omega)=\lim_{n \to \infty} \phi_{s_n+r,t+r}(x,\theta_{-r}\omega)
        $$
        and
        $$
        \tilde{\phi}_{s_n,t}(x,\omega)=\phi_{s_n+r,t+r}(x,\theta_{-r}\omega),
        $$
        so 
        $$
        \lim_{n \to \infty} \tilde{\phi}_{s_n,t}(x,\omega)=\tilde{\phi}_{s,t}(x,\omega)
        $$
        and (iv) follows.
        \item[(v)] follows from \eqref{proof_proposition_nullset} and assumption (v).
        \item[(vi)] Assume that $\omega \in \tilde{\Omega}$ (otherwise there is nothing to prove since 
        $\mP(\tilde{\Omega})=1$). 
        For each $s \geq 0$, there exists a null set $\hat N_s$ in $\tilde\Omega$ such that, for some $r \geq 0$ (possibly depending on $s,t,x,\omega)$, we have $\tilde{\phi}_{s,t}(x,\omega) = \phi_{s+r, t+r}(x,\theta_{-r} \omega) = \phi_{0,t-s}(x,\theta_s \omega)=\phi_{s,t} (x,\omega)$ for $\omega \notin \hat N_s$ and  all $t \geq s$ and $x \in \bar{X}$, so  the  claim 
        follows. 
        \item[(vii)]\hspace{-.15cm}-(x) follow easily from the definitions and \eqref{proof_proposition_nullset}. 
    \end{enumerate}
\end{proof}

\begin{remark}
It is natural to ask if property (vi) in the proposition can be sharpened by stating that 
the processes $(s,t,x)\mapsto \tilde \phi_{s,t}(x)$ and $(s,t,x)\mapsto \phi_{s,t}(x)$ are indistinguishable. We do not know if this holds in general but this is certainly true when $\phi$ only takes values in $X$ since then (iv) can be employed to show that 
the exceptional sets in (vi) can be chosen independently of $s$.
\end{remark}
\begin{theorem}
    \label{existence_RDS_flow}
    Let $\theta=\big(\Omega,\cF,\mP, \{\theta_t\}_{t \in \R}\big) $ be a metric dynamical system and let $(\phi,\Theta)$ be a local semi-flow defined on $(\Omega,\cF,\mP)$ taking values in  a Polish space $X$. Assume that for every $s \geq 0$ there exists a null set
    $N_s$ such that for
   all $x \in X, s\geq 0$ and $\omega \notin N_s$    
    \begin{align}
        \label{explosion_time_invariant}
        \Theta(s,x) (\omega) = \Theta (0,x) (\theta_s \omega) +s
    \end{align}
    and
    \begin{align}
        \label{flow_time_invariant}
        \phi_{s,s+t}(x,\omega)=\phi_{0,t}(x,\theta_s\omega)\quad \text{ for all }  t \in [s, \Theta(s,x)).
    \end{align}
    Then there exist $\tau: X \times \Omega \rightarrow (0, \infty]$ and $\varphi: D  \rightarrow X$ 
    where $$D = \left\{ (t,\omega,x) \in [0,\infty) \times \Omega \times X : t < \tau(x, \omega) \right\}$$
    such that $(\phi_{0,.},\Theta(0,.))$ and $(\varphi, \tau)$ are indistinguishable and  $(\theta,\varphi,\tau)$ is a local RDS. \\
    If $(\phi,\Theta)$ is a continuous/global semi-flow, then $(\varphi,\tau)$ is a continuous/global RDS.
\end{theorem}
\begin{proof}
    Define $\bar{X}:= X \cup \left\{\partial\right\}$  as before. For $0\leq s \leq t$ we set
    \begin{align*}
        \bar{\phi}_{s,t}(x,\omega) = 
        \begin{cases}
            \phi_{s,t}(x,\omega), & \text{if } t<\Theta(s,x)(\omega) \text{ and } x \in X, \\
            \partial, & \text{else}.
        \end{cases}
    \end{align*}
    Then, the continuity assumptions on $\phi$ transfer to $\bar{\phi}$. In particular,
    \begin{itemize}
        \item $s \mapsto \bar{\phi}_{s,t}(x)$ is right continuous on $[0,t)$ if $\bar{\phi}_{s,t}(x)\in X$
        \item $t \mapsto \bar{\phi}_{s,t}(x)$ is right continuous on $[s,\infty)$ for every $x \in \bar{X}$ and $s \geq 0$. 
    \end{itemize}
    By property (v) of the local semi-flow, for all $0\leq s \leq t \leq u$ and $x \in \bar{X}$
    \begin{align*}
        \bar{\phi}_{s,u}(x) = \bar{\phi}_{t,u}(\bar{\phi}_{s,t} (x)).
    \end{align*}
    Using our assumptions \eqref{explosion_time_invariant} and \eqref{flow_time_invariant}, we get
    \begin{align*}
        \bar{\phi}_{s,s+t}(x,\omega)=\bar{\phi}_{0,t}(x,\theta_s\omega)
    \end{align*}
    for all $0\leq s \leq t,x \in \bar{X}$ and $\omega \notin N_s$.\\
    We can now apply Proposition \ref{Proposition_perfection_theorem} and obtain a map $\tilde{\phi}$ 
    satisfying all conclusions in Proposition \ref{Proposition_perfection_theorem} including (vii)
    such that  $\tilde{\phi}_{0,.}(.)$ and $\bar{\phi}_{0,.}(.)$  are indistinguishable. In particular, 
    $\tilde{\phi}$ satisfies the assumptions of the theorem with $N_s=\emptyset$.
    Set $\hat{\varphi}_t(x,\omega): = \tilde{\phi}_{0,t}(x,\omega)$. Then $\hat{\varphi}$ satisfies the cocycle property
    \begin{align*}
       \hat{\varphi}_{s+t}(x,\omega) &= \tilde{\phi}_{0,s+t}(x,\omega)
       = \tilde{\phi}_{s,s+t} ( \tilde{\phi}_{0,s} (x,\omega), \omega) \\
       &= \tilde{\phi}_{0,t} ( \tilde{\phi}_{0,s} (x,\omega), \theta_s\omega)
       = \hat{\varphi}_{t}(\hat{\varphi}_{s}(x,\omega),\theta_s\omega)
    \end{align*}
    for any $s,t\geq 0$, $x \in \bar{X}$ and $\omega \in \Omega$.\\
    Set $\tau (x,\omega) = \inf \left\{ t> 0 : \hat\varphi_t (x,\omega) = \partial \right\}$ and $D= \left\{ (\omega,t,x) \in \Omega \times [0,\infty) \times X: t < \tau(x,\omega) \right\}$.  \\
    Let $\varphi: = \hat{\varphi}|_D$. Then $(\theta,\varphi,\tau)$ is a local random dynamical system on $X$ and the indistinguishability statement follows from the construction. The final statements in the theorem are clear.
\end{proof}

\section{Application to SDEs}

We can use the previous theorem to show that SDEs that admit a (local) stochastic flow also generate a (local) RDS.

Let us repeat the basic strategy: if an SDE (possibly infinite dimensional, e.g.~an SPDE or a stochastic delay equation) with time-independent coefficients driven by a process with stationary increments has unique local solutions, then, in many cases, the solutions can be shown to admit a local 
or global continuous semi-flow to which the previous theorem can be applied, so they generate an RDS. 
Examples of results of this kind are contained in \cite{AS} where the processes driving a finite dimensional SDE 
were continuous semimartingales (with stationary increments) and \cite[Theorem 5]{KS}, in which the 
driving semimartingales were just c\`adl\`ag (i.e.~right continuous with left limits) but the coefficients of the SDE were assumed to satisfy a global Lipschitz condition (thus excluding the possibility of blow-up). For simplicity we now restrict to finite dimensional equations.\\
In the following we assume that $Z^1, \dots , Z^m$ are real-valued c\`adl\`ag semimartingales and that $f^k: \mathbb{R}^d \rightarrow \mathbb{R}^d$ are measurable for $1\leq k\leq m$.

\begin{definition}
    We say that that an SDE 
     \begin{align}
        \label{SDE}
        \text{d} X_t = \sum_{k=1}^m f^k(X_{t-}) \text{ d}Z^k_t
    \end{align}
    has a \emph{(strong) local solution} if for each $x \in \mR^d$ and $s \geq 0$, there exists a stopping time $\tau= \tau_{s,x}>s$ and an $\mR^d$-valued adapted process $X_t$, $t \in [s, \tau)$ with c\`adl\`ag paths such that $X_t = x + \sum_{k=1}^m \int_s^t f^k(X_{u-}) \text{d}Z^k_u $ almost surely whenever $s \leq t < \tau$ and 
    $\limsup_{t \rightarrow \tau} |X_t| = \infty$ almost surely on the set $\left\{ \tau < \infty \right\}$. \\
    We say that the local solution is \emph{unique} if whenever $\tilde{X}_t$, $t\geq s$ is another process with these properties with associated stopping time $\tilde{\tau}$, then $\tau = \tilde{\tau}$ and $X = \tilde{X}$ on $[s,\tau)$ almost surely. \\
    We say that \eqref{SDE} admits a \emph{local semi-flow}, if it has a unique local solution which admits a modification $(\phi, \Theta)$ which is a local semi-flow. In particular, $\limsup_{t \rightarrow \Theta(s,x)} | \phi_{s,t} (x)| = \infty$ whenever $s \geq 0, x \in \mR^d$ and $\Theta (s,x) < \infty$.\\
\end{definition}

\begin{theorem}\label{SDERDS}
    Let $Z^1, \dots , Z^m$ be real-valued c\`adl\`ag semimartingales with stationary increments in the sense that the law of $Z^k_{t+h} - Z^k_t,\,h\geq 0,\,k\in\{1,...,m\}$ does not depend on $t$. Moreover let $f^k: \mathbb{R}^d \rightarrow \mathbb{R}^d$ be measurable for $1\leq k\leq m$ .
    If the  equation
    \begin{align*}
        \dif X_t = \sum_{k=1}^m f^k(X_{t-}) \dif Z^k_t
    \end{align*}
    has a unique strong local solution which admits a continuous local semi-flow, then it also admits a continuous local RDS.\\
    If the semi-flow  is even global, then it admits a (global) RDS.
\end{theorem}
\begin{proof}
    All we have to do is to transfer the local semi-flow to a suitable MDS and to check that conditions
    \eqref{explosion_time_invariant} and \eqref{flow_time_invariant} hold. This is done similarly as in
    \cite{AS} and in \cite[Section 3]{KS}.
    
    We can and will assume that  $Z^1, \dots , Z^m$ are defined on $\R$ and are 0 at 0.
    
    \emph{Step 1: Definition of MDS}\\
    Let $\Omega= \mathcal{D}_0(\mathbb{R},\mathbb{R}^m)$ be the set of c\`adl\`ag functions from $\mathbb{R}$ to $\mathbb{R}^m$ which are $0$ at $0$ 
    and equipped with the (local) Skorokhod topology and let $\mathcal{F}$ be the Borel-$\sigma$-algebra of $\Omega$. Set $Z=(Z^1, \dots, Z^m)$ and let $\mathbb{P}$ be the law of $(Z_t)_{t\in \mathbb{R}}$. Define a shift $\theta$ on $\Omega$ by
    \begin{align*}
        (\theta_t \omega) (s) = \omega(t+s) - \omega(t)
    \end{align*}
    for all $\omega \in \Omega$ and $t,s\in \mathbb{R}$. We  now define the process $Z$ on the MDS $\theta$ by $Z^k_t(\omega):=\omega_k(t),\, t\in \R, \,k\in \{1,...,m\}$ (we will not use a new symbol).\\
    \emph{Step 2: Time-invariance of stochastic flow}\\
    To apply  Theorem \ref{existence_RDS_flow}, it only remains to show that equations \eqref{explosion_time_invariant} and \eqref{flow_time_invariant} hold for the continuous local semi-flow  $(\phi_{s,t}(x), \Theta(s,x))$ generated by the SDE on the MDS $\theta$.
    Fix $s\geq 0$ and $x\in \mathbb{R}^d$. Then 
    \begin{align}
        \label{proof_solution1}
        \phi_{s,s+t}(x,\omega)= x + \sum_{k=1}^m \left( \int_{(s,s+t]} f^k ( \phi_{s,u-}(x,.))\, \dif Z^k_u \right) (\omega)
    \end{align}
    almost surely for all $t\in [0, \Theta(s,x)(\omega)-s)$. Further, we have
     \begin{align}
        \label{proof_solution2}
        \phi_{0,t}(x,\theta_s\omega)= x + \sum_{k=1}^m \left( \int_{(0,t]} f^k ( \phi_{0,u-}(x,.))\, \dif Z^k_u \right) (\theta_s\omega)
    \end{align}  
    almost surely for all $t\in [0, \Theta(0,x)(\theta_s\omega))$.
    We want to show that
     \begin{align}
     \label{integral_unique_solution}
        \phi_{s,s+t}(x,\omega)= x + \sum_{k=1}^m \left( \int_{(0,t]} f^k ( \phi_{s,(s+u)-}(x,.)) \text{ d}Z^k_u(\theta_s(.)) \right) (\omega)
    \end{align}
    almost surely for all $t\in [0, \Theta(s,x)(\omega)-s)$. 
    Afterwards we can use the uniqueness of the solution.
    Comparing the right-hand side of \eqref{proof_solution1} and \eqref{integral_unique_solution}, \eqref{integral_unique_solution} holds true if
    \begin{align}
        \label{integral}
        \left(\int_{(s,s+t]} g^k(u,.) \text{ d}Z^k_u \right) (\omega)
        = \left( \int_{(0,t]} g^k(u+s, .) \text{ d}Z^k_u (\theta_s(.)) \right) (\omega)
    \end{align}
    for each $1 \leq k \leq m$ where $g^k(u, \omega)= f^k ( \phi_{s,u-} (x,\omega))$. One can check that \eqref{integral} is true by first checking it for a simple predictable function $g$ and then using an approximation argument. Alternatively, we can derive \eqref{integral} from \cite[Theorem 3.1 (vi)]{Pro}.
    \\
    Comparing  \eqref{proof_solution2} and \eqref{integral_unique_solution}, uniqueness of the local solution implies that $\Theta(s,x)(\omega)=\Theta(0,x) (\theta_s \omega)+s$ and  $\phi_{s,s+t}(x,\omega) =\phi_{0,t}(x,\theta_s \omega)$ for all $t\in [0,\Theta(s,x)(\omega)-s)$ and $\omega \notin N_{s,x}$. Here $N_{s,x}\in \mathcal{F}$ is a null set that may depend on $s$ and $x$. Indeed, choosing $N_s= \bigcup_{x \in \mathbb{Q}^d} N_{s,x}$, equation \eqref{flow_time_invariant} holds since $\phi$ is continuous in $x$. Then, for each $s\geq 0$ we have
    \begin{align*}
        \limsup_{t \rightarrow \Theta(s,x)(\omega) -s} |\phi_{0,t} (x, \theta_s \omega ) |
        =  \limsup_{t \rightarrow \Theta(s,x)(\omega)-s} |\phi_{s, s+t} (x, \omega ) | = \infty
    \end{align*}
    for all $x \in \mR^d$ and $\omega \notin N_s$. By uniqueness of the explosion time, equation \eqref{explosion_time_invariant} follows. Applying Theorem \ref{existence_RDS_flow} we obtain the claimed RDS.
\end{proof}

\begin{remark}
Let us emphasize that if an SDE generates a continuous local semi-flow and if the SDE admits a global 
solution for every initial condition then this does {\em not} imply that the local semi-flow is global. A counterexample of an SDE in $\R^2$ with bounded and infinitely differentiable coefficients can be found in \cite{LS}. SDEs which generate a  continuous local semi-flow and admit a global 
solution for every initial condition are often referred to as {\em weakly complete} while those generating a continuous global semi-flow are called {\em strongly} or {\em strictly complete} 
(sometimes this term is used if the solution map $\phi_{0,t}(x)$ admits a modification which is 
jointly continuous in $(t,x)$ which is slightly weaker than assuming that it generates a continuous global semi-flow, see \cite[Definition 1.3]{SS}). To prevent confusion about varying definitions of 
strong or weak completeness we will avoid these terms in what follows.

\end{remark}

\section{Examples of (local) RDS generated by  SDEs }
In this section we consider the following SDE on $\R^d$ with time homogeneous coefficients
\begin{align}\label{sde01}
\dif X_t=b(X_t)\,\dif t+\sigma(X_t)\,\dif W_t,\quad X_s=x\in\mathbb{R}^d,\quad t\geq s \geq 0,
\end{align}           
where $d\geq1$, $b: \mR^d\rightarrow\mR^d $ and $\sigma=(\sigma_{ij})_{1\leq i,j\leq d}: \mR^d\rightarrow L(\mathbb{R}^d)$ $(:=d\times d$ real valued matrices$)$ are measurable, and $(W_t)_{t\geq0}$ is a standard $d$-dimensional Brownian motion defined on some filtered probability space $(\Omega,\mathcal{F},(\mathcal{F}_t)_{t\geq 0},\mP)$. We denote the Euclidean norm on $\R^d$ by $|.|$ and the induced norm on $L(\R^d)$ by $\|.\|$. Further, $\langle.,.\rangle$ denotes the standard inner product on $\R^d$. Recall that
 the trace of $a:=(a_{ij})_{1\leq i,j\leq d}:=\sigma\sigma^*$ satisfies $\mathrm{tr}(a)=\sum_{i,j=1}^d\sigma_{ij}^2$, where $\sigma^*$ denotes the transpose of $\sigma$. 
\subsection{RDS generated by SDEs with locally monotone coefficients}

In this subsection we assume that, in addition, $b$ and $\sigma$ are continuous.
The following facts were established in \cite{SS}.

\begin{theorem}[{\cite[Theorem 2.4]{SS}}]\label{SS1}
Assume that there exists some $\mu>d+2$ such that, for each $R>0$, there exists some $K_R$ such that 
\begin{align*}
2\langle b(x)-b(y),x-y\rangle&+ \mathrm{tr}[(\sigma(x)-\sigma(y))(\sigma(x)-\sigma(y))^*]\\&+\mu\Vert \sigma(x)-\sigma(y)\Vert^2\leq K_R|x-y|^2\quad \text{ for all } |x|,|y|\leq R.
\end{align*}
Then the SDE \eqref{sde01} admits a continuous local semi-flow.
\end{theorem}

\begin{theorem}[{\cite[Theorem 2.5]{SS}}]\label{SS2}
Assume that the following two conditions hold.
\begin{itemize}
    \item [(1)] There exists a non-decreasing function $g:[0,\infty)\rightarrow(0,\infty)$ 
    such that $\int_0^\infty\frac{1}{g(x)}\dif x=\infty$ and 
 $$ 
 2 \langle b(x),x\rangle +\mathrm{tr}\big[\sigma(x)\sigma^*(x)\big]\leq g\big(|x|^2\big),  
 $$   
 for all $x\in \R$. 
 \item [(2)] There exist a continuous and nondecreasing function $f:[0,\infty) \to (0,\infty)$ and
 $\mu>d+2$ such that for any $R>0$ and $|x|,|y| \leq R$ we have
 \begin{align}\label{conf}
2\langle b(x)-b(y),x-y\rangle &+ \mathrm{tr}[(\sigma(x)-\sigma(y))(\sigma(x)-\sigma(y))^*]\\&+\mu \Vert \sigma(x)-\sigma(y) \Vert^2\leq f(|x|\vee|y|)|x-y|^2
\end{align}
 and there exist $\gamma>0$ and $t_0>0$ such that, for each $R>0$,
 $$
\sup_{|x|\leq R}\sup_{s\in[0,t_0]}\mE e^{\gamma f(|X_{0,t}^x|)}<\infty,
 $$
 where $X_{0,t}^x$ denotes the solution of the SDE starting from $x\in \R^d$.
\end{itemize}

Then the SDE \eqref{sde01} admits a continuous global semi-flow.
\end{theorem}
Now we can apply Theorem \ref{SDERDS} and obtain the following result.

\begin{corollary}
Under the assumptions of Theorem \ref{SS1} respectively \ref{SS2}, equation \eqref{sde01} generates 
a local respectively a global RDS.
\end{corollary}
\begin{remark}\label{rmnt}
Existence and uniqueness of strong local solutions to \eqref{sde01} were studied in \cite[Theorem 3.1.1]{LWMR} under more general monotonicity conditions. \cite{SS} also study more general SDEs which are
driven by Kunita-type Brownian flows. They also provide three examples of explicit conditions on the coefficients $b,\sigma$ and the function $f$ such that the assumptions of Theorem \ref{SS2} hold, 
\cite[Proposition 2.3]{SS}. If, for example, $b$ and $\sigma$ are globally bounded, then one can choose 
$f(u)=\beta \big(u^2+1\big)$ for an arbitrary $\beta>0$.
\end{remark}

\subsection{Existence of RDS for singular SDEs}
 Unlike the previous example we now focus on  singular SDEs for which the drift $b$ is not pointwise defined but instead satisfies some integrability condition. The price we have to pay is a non-degeneracy assumption on the diffusion part which is needed to ensure the existence of solutions.  There are many  papers on singular SDEs  studying its well-posedness and properties of the solutions, see e.g. \cite{KR},  \cite{VE}, \cite{XXZZ}, \cite{Zhang2017}, \cite{XieZhang}, \cite{Zv}.  Our work seems to be the first one showing that, under appropriate conditions, such equations generate an  RDS. 
 
 We first introduce some notation for later use. For $p\in[1,\infty)$, let $\mathbb{L}_p(\mathbb{\mR}^d)$ denote the space of all real Borel measurable functions on $\mathbb{R}^d$ equipped with the norm
$$\Vert f\Vert_{\mathbb{L}_p}:=\Big(\int_{\mathbb{R}^d}|f(x)|^p\,\dif x\Big)^{1/p}<+\infty.$$
We also introduce the notion of  a localized $\mL_p$-space: for fixed $r>0$,
$$\tilde\mL_p(\mathbb{R}^d):=\{f:\Vert f\Vert_{\tilde\mL_p}:=\sup_{z}\Vert \chi_r^zf\Vert_{\mL_p}<\infty\}$$
 where  $\chi_r(x):=\chi(\frac{x}{r})$ and  $\chi_r^z(x):=\chi_r(x-z)$, $\chi\in C_c^\infty(\mR^d)$ is a smooth function with $\chi(x)=1$ for $|x|\leq1$,  and $\chi(x)=0$ for $|x|>2$.  For further studies on these spaces we refer to \cite{XXZZ}. In the following, all derivatives should be interpreted in the weak sense.

We first state our result on the  existence of a global semi-flow for singular SDEs.
 \begin{theorem}
\label{sSDE}For $i=1,2$, assume $\frac{d}{p_i}<1$ with $p_i\in[2,\infty)$ and 
\begin{itemize}
\item[(1)] $|b|\in \tilde\mL_{p_1}(\mathbb{R}^d)$,  $|\nabla\sigma|\in \tilde\mL_{p_2}(\mathbb{R}^d)$,
\item [(2)] there exists some $K\geq 1$ such that for all $x\in\mathbb{R}^d$, $$K^{-1}|\xi|^2\leq\langle\sigma\sigma^*(x)\xi,\xi\rangle\leq K|\xi|^2,\quad \forall \xi\in\mathbb{R}^d,  $$  and $a:=\sigma\sigma^*$  is uniformly continuous in $x\in\mathbb{R}^d$.
    \end{itemize}
    Then  the SDE \eqref{sde01} admits a continuous global semi-flow.
 \end{theorem}
 \begin{proof}
 The result follows essentially  by combining and extending \cite{FF}, \cite{AVS}, and \cite[Theorem 1.1]{XXZZ} (the latter paper establishes weak differentiability of the solution map with respect to the spatial variable). 
 We therefore simply give an outline of the proof.  For fixed $T>0$, we define the following  Zvonkin transformation map
\begin{align}\label{pdephi}\Phi^b(t,x):=x+U_b(t,x) \text{ for } (t,x)\in[0,T]\times\mathbb{R}^{d},\end{align}
where  $U_b:=(u_b^{(l)})_{1\leq l\leq d}$ and $u_b^{(l)}$ are the unique solutions to the equations
$$
\partial_tu^{(l)}+\frac{1}{2}\sum_{i,j=1}^da_{ij}\partial_{ij}^2u^{(l)}+b\cdot\nabla u^{(l)}+b^{(l)}=\lambda u^{(l)}, \quad t\in[0,T],\quad u^{(l)}(T,x)=0,\quad l=1,\cdots,d
$$
such that $\partial_tu_b^{(l)},\nabla u_b^{(l)},\nabla^2u_b^{(l)}, u_b^{(l)}$ all are belonging to $\mL_q([0,T],\tilde\mL_{p_1}(\mathbb{R}^d))$ for some $q\in[2,\infty)$ and $\frac{d}{p_1}+\frac{2}{q}<1$.
Then it is known that for sufficiently large $\lambda>0$ the map $\Phi^b(t, \cdot)$ is  a $\mathcal{C}^1$-diffeomorphism on $\mathbb{R}^d$ (see, e.g.,  \cite[p.5205]{XXZZ}). Therefore  it is enough to show for the following transformed equation
 \begin{align}\label{SDEY}\dif Y_t=\tilde b(t,Y_t)\,\dif t+\tilde\sigma(t,Y_t)\,\dif W_t, \quad Y_0=y\in\mathbb{R}^d
  \end{align}
  with
  $$\tilde b(t,x):=\lambda U_b(t,\Psi^b(t,x)),\tilde\sigma(t,x):=[\nabla\Phi^b\cdot\sigma]\circ({\Psi^b}(t,x)),\Psi^b:=(\Phi^b)^{-1},\quad y=\Phi^b(0,x_0),$$
that there is a flow $(s,t,y,\omega)\mapsto \zeta_{s,t}(y,\omega)$, defined for $0\leq s\leq t\leq T$, $y\in\mathbb{R}^d$ and $\omega\in\Omega$ with values in $\mathbb{R}^d$, such that for any $s\in[0,T]
$\begin{itemize}
\item[(1')] for any $y\in\mathbb{R}^d$, the process $Y^{s,y}=\{Y^{y}_{s,t}:=\zeta_{s,t}(y):t\in[s,T]\}$ is a continuous solution of \eqref{SDEY},
\item[(2')]  $\zeta_{s,t}(y)=\zeta_{u,t}(\zeta_{s,u}(y))$ for all $0\leq s\leq u\leq t\leq T$ and $y\in\mathbb{R}^d$ and $\zeta_{s,s}(y)=y,$
\item[(3')]  for $\alpha\in(0,1)$  there exists  $C(\alpha, \omega,T)>0$ such that for any $x,y\in\mathbb{R}^d$  and $0\leq s\leq t\leq T$,
    $$|\zeta_{s,t}(x)-\zeta_{s,t}(y)|\leq C|x-y|^\alpha.$$
\end{itemize}
By  \cite[Theorem 1.1 (B)]{XXZZ} we know that there exists a unique global strong solution $(\mY_{s,t}^y)_{t\geq s}$ to \eqref{SDEY} with $\mY_s=y\in\mathbb{R}^d$. For property (3') we use the same argument from \cite[(4.10)]{XXZZ} to get for any $\alpha'\geq 2$ there exists a constant $C=C(\alpha',T)$ such that for any $x,y\in\mR^d$ such that
\begin{align*}
\mE\sup_{t\in[0,T]}|\mY_{0,t}^x-\mY_{0,t}^y|^{\alpha'}\leq C|x-y|^{\alpha'}.
\end{align*}
Furthermore, by the boundedness of $\tilde b$ and $\tilde \sigma$ (see \cite[Proof of Theorem 1.1]{XXZZ}), we can easily obtain for $0\leq s\leq t\leq t'\leq T$,
\begin{align*}
\sup_{x\in\mR^d}\mE|\mY_{s,t'}^x-\mY_{s,t}^x|^{\alpha'}\leq C_{\alpha'}|t'-t|^{\frac{\alpha'}{2}}.
\end{align*}
Therefore for all $x,y\in\mR^d$ and all $0\leq s\leq t\leq T$, $0\leq s'\leq t'\leq T$, there exist positive constants $c_{\alpha}'$ and $c_{\alpha}''$ such that
\begin{align*}
\mE|\mY_{s,t}^x-\mY_{s',t'}^y|^{\alpha'}\leq & c_\alpha'\Big(\mE|\mY_{s,t}^x-\mY_{s,t}^y|^{\alpha'}+\mE|\mY_{s,t}^y-\mY_{s,t'}^y|^{\alpha'}+\mE|\mY_{s,t'}^y-\mY_{s',t'}^y|^{\alpha'}\Big)
\\\leq &c_\alpha''(|x-y|^{\alpha'}+|t-t'|^{\frac{\alpha'}{2}}+|s-s'|^{\frac{\alpha'}{2}}).
\end{align*}
The last inequality holds because for all $x\in\mR^d$, for all $s, s'\leq t$,  w.l.o.g. assuming $s\leq s'$, 
\begin{align*}
   \mE|\mY_{s',t}^x-\mY_{s,t}^x|^{\alpha'}= &\mE|\mY_{s',t}^x-\mY_{s',t}^{\mY_{s,s'}^x}|^{\alpha'}
   \\=&\mE[\mE|\mY_{s',t}^x-\mY_{s',t}^{\mY_{s,s'}^x}|^{\alpha'}|\mathcal{F}_{s'}]
   \\\leq & C\mE[|\mY_{s,s'}^x-x|^{\alpha'}]\leq C|s-s'|^{\frac{\alpha'}{2}}.
\end{align*}
Finally  the existence of the uniformly H\"older continuous flow $\zeta_{s,t}$ with property (3') follows from the Kolmogorov continuity theorem (for detailed argument, see e.g. \cite[proof of Proposition 5.2.]{SS}). Then for such $\zeta_{s,t}$ property (1') holds obviously.  The uniqueness of strong solution to \eqref{SDEY} implies (2'), see, e.g. \cite[Theorem 1.1]{XXZZ}.
 \end{proof}
 \begin{remark}
   We mention that the global semi-flow $\psi$ in Theorem \ref{sSDE} enjoys the additional property that $\psi_{s,t}(.,\omega):\mathbb{R}^d\rightarrow\mathbb{R}^d$ is a homeomorphism. This  can be shown in the same way as in the proof of \cite[Theorem 5.10]{FF} (for equations with additive noise): first  observe that it holds for  the semi-flow $\zeta$ generated by the transformed SDE \eqref{SDEY} and then, by properties  of the Zvonkin transformation,  it also holds for $\psi$.
 \end{remark}
 \begin{example}
  Here are two typical examples which satisfy the conditions in Theorem \ref{sSDE}.
 \begin{itemize}
     \item[(1)] $|b|\in\mL_{p_1}(\mathbb{R}^d),|\nabla \sigma|\in\mL_{p_2}(\mathbb{R}^d)$ with $p_i\in[2,\infty)$ such that $\frac{d}{p_i}<1$, $i=1,2$,  $\sigma$ further satisfies (2) in Theorem \ref{sSDE}.
     \item[(2)]  $b$ and $|\nabla \sigma|$ are bounded (see \cite[Remark 1.2]{XXZZ}), $\sigma$ further satisfies (2) in Theorem \ref{sSDE}. \\
     This case may be compared with the example at the end of Remark \ref{rmnt} in the previous section in which $b$ and $\sigma$ were also assumed to be bounded. Notice however that the additional assumptions were different: in Remark \ref{rmnt} we assumed a monotonicity condition to hold while here we assume non-degeneracy of the diffusion.  
 \end{itemize}
 \end{example}
 \begin{theorem}\label{lsSDE}
 For any $n\in\mathbb{N}$ and some $p_n\in[2,\infty)$ such that $\frac{d}{p_n}<1$, assume 
\begin{itemize}
\item[(1)] $|b|\in \tilde\mL_{p_n}(B_n)$ where $B_n:=\{x\in\mathbb{R}^d:|x|\leq n\}$;
\item [(2)] $|\nabla\sigma|\in \tilde\mL_{p_n}(B_n)$ and there exists $K_n\geq 1$ such that for all $x\in B_n$, $$K_n^{-1}|\xi|^2\leq\langle\sigma\sigma^*(x)\xi,\xi\rangle\leq K_n|\xi|^2,\quad \forall \xi\in\mathbb{R}^d,  $$  and $a:=\sigma\sigma^*$  is uniformly continuous on $B_n$.
    \end{itemize}
    Then  the SDE \eqref{sde01} admits a continuous local semi-flow.
 \end{theorem}
 \begin{proof}
We adapt the localization technique (e.g. see \cite[Proof of Theorem 1.3]{Zhang2011}, \cite[Proof of Theorem 2.4]{SS}) to show the existence of a local semi-flow. Denote for $x\in\mathbb{R}^d$,
\begin{align*}
b_n(x)&=b(x)\mathbbm{1}_{B_n}(x),\\
\sigma_n(x)&=\sigma\Big(z_{n+1}+\chi_n(x)(x-z_{n+1})\Big), \text{  for some }z_{n+1}\in B_{n+1},
\end{align*}
where  $\chi_n:\mathbb{R}^{d}\rightarrow[0,1]$ is a  nonnegative smooth function   such that $\chi_n(y)=1$ for all $y\in  B_n$ and $\chi_n(y)=0$ for all $y\notin  B_{n+1}$, $B_{n}:=\{x\in\mR^d:|x|<n\}$, $n\in\mN$.  It can be verified that such $b_n$ and $\sigma_n$ fulfill the conditions in Theorem \ref{sSDE}.  Therefore there exists a global continuous semi-flow $\psi^n$ for each $n\in\mathbb{N}$ following from Theorem \ref{sSDE}. 
We define 
\begin{align*}
    \Theta(s,x,\omega):&=\lim_{n\rightarrow\infty}\{t\geq s:|\psi_{s,t}^n(x)|\geq n\},
    \\\psi_{s,t}(x,\omega):&=\psi_{s,t}^n(x) \text{ for } t\in[s, \Theta(s,x,\omega)).
\end{align*}
It is clear that $(\psi,\Theta)$ is a local semi-flow of \eqref{sde01} following the argument from  \cite[Proof of Theorem 2.4]{SS}.
\end{proof}

 \begin{remark}
  Since our main interest are SDEs generating an RDS, we study singular SDEs with time homogeneous   coefficients only. For the existence of a semi-flow, actually the above local/global results hold also when the coefficients are  not time homogeneous. Then the required conditions are:  $|b|\in \mL_{q_1}([0,T],\tilde\mL_{p_1}(\mathbb{R}^d))$, $|\nabla\sigma|\in \mL_{q_2}([0,T],\tilde\mL_{p_2}(\mathbb{R}^d))$ for some $p_i,q_i\in[2,\infty)$ with $\frac{d}{p_i}+\frac{2}{q_i}<1$ (known as 'LPS' condition), $i=1,2$  for the existence of a global semi-flow.  For the existence of local semi-flow we assume $|b|,|\nabla\sigma| \in \mL_{q_n}([0,n],\mL_{p_n}(B_n))$ for $p_n,q_n\in[2,\infty)$ with $\frac{d}{p_n}+\frac{2}{q_n}<1$, $n\in\mN$. In both cases  it is necessary to assume the uniform ellipticity condition to hold globally and locally respectively. 
 \end{remark}
 With Theorems \ref{sSDE} and \ref{lsSDE} at hand, we are ready to apply Theorem \ref{SDERDS} to 
 conclude the following result.

\begin{corollary}
Under the assumptions of Theorem \ref{sSDE} respectively \ref{lsSDE}, equation \eqref{sde01} generates 
a  global respectively a local RDS.
\end{corollary}
 
\section*{Acknowledgments}
\footnotesize Financial support for C. Ling by the DFG through the research unit (Forschergruppe) FOR 2402 is acknowledged.

	\end{document}